\documentclass{article}

\usepackage{PRIMEarxiv}
\usepackage[utf8]{inputenc}
\usepackage[T1]{fontenc}
\usepackage[english]{babel}
\usepackage{xcolor}
\usepackage{stmaryrd}
\usepackage{amssymb}
\usepackage{amsmath}
\usepackage{libertine}
\usepackage{graphicx}
\usepackage{xspace}
\usepackage{amsthm}
\usepackage{epsfig}
\usepackage{latexsym}
\usepackage{caption}
\usepackage{amsfonts}
\usepackage{mathrsfs}
\usepackage{enumerate}
\usepackage{graphics}
\usepackage{float}
\usepackage{pict2e}
\usepackage{subfigure}
\usepackage{booktabs}
\usepackage{algorithm}
\usepackage{algorithmic}
\usepackage{hyperref}
\hypersetup{
    colorlinks=true,
    linkcolor=blue,
    filecolor=magenta,      
    urlcolor=cyan,
    citecolor=blue,
}

\newcommand{\chic}{\chi_c}
\newcommand{\boxprod}{\square}
\newcommand{\boxprodprime}{\square'}

\newtheorem{theorem}{Theorem}

\newtheorem{corollary}[theorem]{Corollary}
\newtheorem{definition}[theorem]{Definition}
\newtheorem{conjecture}[theorem]{Conjecture}

\newtheorem{problem}[theorem]{Problem}

\newtheorem{lemma}[theorem]{Lemma}
\newtheorem{proposition}[theorem]{Proposition}

\newtheorem{example}[theorem]{Example}
\newtheorem{remark}[theorem]{Remark}
\title{Circular Chromatic Number of Cartesian Product of Signed Graphs}

\author{
  Ebode Atangana Pie Desire \\
  Faculty of Sciences\\
  University of Yaounde 1 \\
  Yaounde, Cameroon \\
  \texttt{desire.ebode@univ-yaounde1.cm} 
}

\begin{document}

\maketitle

\begin{abstract}
This paper studies circular coloring of signed graphs. A signed graph is a graph with a signature that assigns a sign to each edge, either positive or negative. We investigate circular coloring and the circular chromatic number for two types of Cartesian products of signed graphs. Our main results are: (1) For the Cartesian product Type~1 $(G,\sigma)\boxprod (H,\tau)$, we have $\chic(G \boxprod H,\sigma\boxprod\tau)=\max\{\chic(G,\sigma),\chic(H,\tau)\}$; (2) For the Cartesian product Type~2 $(G,\sigma)\boxprodprime (H,\tau)$, we have $\chic(G \boxprodprime H,\sigma\boxprodprime \tau)\leq 2\max\{\chic(G),\chic(H)\}$, and this bound is asymptotically tight.
\end{abstract}


\section{Introduction}
The study of signed graphs originates from the pioneering work of Harary \cite{harary1953notion}, who introduced this mathematical framework to model social structures in psychology. In social networks, relationships between individuals are often categorized as positive (friendship, alliance), negative (animosity, conflict), or neutral. Graph theory provides a natural language for such representations: vertices represent individuals, and edges represent relationships between them. By assigning signs to edges, we obtain a richer structure that captures the qualitative nature of interactions beyond mere existence.

\begin{definition}
A \emph{signed graph} is an ordered pair $(G,\sigma)$, where:
\begin{itemize}
    \item $G = (V,E)$ is an undirected graph (possibly with multiple edges) called the \emph{underlying graph},
    \item $\sigma: E(G) \to \{+1,-1\}$ is a function called the \emph{signature} or \emph{sign function}.
\end{itemize}
An edge $e \in E(G)$ with $\sigma(e) = +1$ is called a \emph{positive edge}, often denoted by a solid line, while an edge with $\sigma(e) = -1$ is called a \emph{negative edge}, often denoted by a dashed line.
\end{definition}

The sign of an edge represents the nature of the relationship between its endpoints: positive for friendship/alliance, negative for hostility/conflict. This simple addition of signs leads to rich combinatorial structure with applications in social psychology \cite{cartwright1956structural}, computer science \cite{broy1992software}, and physics \cite{zaslavsky1982signed}.

\begin{definition}
Given a signed graph $(G,\sigma)$ and a vertex $v \in V(G)$, \emph{switching at $v$} produces a new signed graph $(G,\sigma')$ where:
$$\sigma'(e) = \begin{cases}
-\sigma(e) & \text{if $e$ is incident to $v$}\\
\sigma(e) & \text{otherwise.}
\end{cases}$$
More generally, switching at a subset $X \subseteq V(G)$ changes the sign of every edge in the cut $(X, V(G)\setminus X)$.
\end{definition}

Switching represents changing one's perspective in a social network: if we switch our opinion about a person from positive to negative (or vice versa), all our relationships with that person flip sign. Two signed graphs $(G,\sigma)$ and $(G,\tau)$ are \emph{switching equivalent} if one can be obtained from the other by a sequence of switching operations. Switching equivalence is a fundamental equivalence relation in signed graph theory, as many important properties are invariant under switching.

\begin{definition}
A cycle $C = v_1v_2\cdots v_kv_1$ in a signed graph $(G,\sigma)$ is \emph{balanced} if the product of signs along the cycle is positive:
$$\prod_{i=1}^k \sigma(v_iv_{i+1}) = +1,$$
where indices are taken modulo $k$. Equivalently, a cycle is balanced if it contains an even number of negative edges. A cycle that is not balanced is called \emph{unbalanced}. A signed graph is \emph{balanced} if all its cycles are balanced.
\end{definition}

The concept of balance formalizes the social psychology principle that "the friend of my friend is my friend" and "the enemy of my enemy is my friend." Harary \cite{harary1953notion} proved the fundamental characterization: a signed graph is balanced if and only if its vertex set can be partitioned into two parts such that all positive edges are within parts and all negative edges are between parts. Moreover, a signed graph is balanced if and only if it is switching equivalent to the all-positive graph $(G,+)$ \cite{zaslavsky1982signed}.

Beyond social network analysis, signed graphs appear in numerous mathematical and computational contexts: matroid theory \cite{zaslavsky1982signed}, knot theory, statistical physics (spin glasses), coding theory, and constraint satisfaction problems. Many classical graph theory concepts have been extended to signed graphs, including connectivity \cite{zaslavsky1981characterizations}, minors \cite{zaslavsky2013matrices}, flows \cite{Raspaudcircular}, and graph homomorphisms \cite{naserasr2015homomorphisms,naserasr2021homomorphisms}.

\subsection{Circular Coloring of Signed Graphs}
Graph coloring is one of the most fundamental topics in graph theory, with applications in scheduling, register allocation, and frequency assignment. The \emph{circular chromatic number}, introduced by Vince \cite{vince1988star} (who called it the \emph{star-chromatic number}), generalizes the classical chromatic number by allowing colors to be placed on a circle. This concept captures fractional coloring and provides a refined measure of coloring complexity.

For unsigned graphs, a \emph{circular $r$-coloring} ($r \ge 2$) assigns to each vertex a point on a circle of circumference $r$, such that adjacent vertices receive points at circular distance at least 1. The \emph{circular chromatic number} $\chi_c(G)$ is the infimum of all $r$ for which such a coloring exists. It is known that $\chi(G) - 1 < \chi_c(G) \le \chi(G)$ for all graphs $G$, and $\chi_c(G)$ can be any rational number in this interval \cite{zhu2001circular}.

The natural extension of circular coloring to signed graphs was introduced by Naserasr, Wang, and Zhu \cite{naserasr2020circular}. For signed graphs, we need to account for the different constraints imposed by positive and negative edges:

\begin{definition}
Let $r \ge 2$ be a real number, and let $C^r$ denote the circle of circumference $r$, obtained by identifying the endpoints of the interval $[0,r]$. For $x \in C^r$, let $\overline{x} = x + r/2 \pmod{r}$ denote its \emph{antipodal point}. A \emph{circular $r$-coloring} of a signed graph $(G,\sigma)$ is a function $f: V(G) \to C^r$ satisfying:
\begin{itemize}
    \item For each positive edge $uv \in E(G)$: $d_{C^r}(f(u), f(v)) \ge 1$
    \item For each negative edge $uv \in E(G)$: $d_{C^r}(f(u), \overline{f(v)}) \ge 1$
\end{itemize}
where $d_{C^r}(x,y) = \min\{|x-y|, r - |x-y|\}$ is the circular distance on $C^r$.
\end{definition}

Intuitively, positive edges require their endpoints to be far apart (distance at least 1), while negative edges require one endpoint to be far from the antipodal of the other. This formulation elegantly captures both types of constraints in a unified geometric framework.

\begin{definition}
The \emph{circular chromatic number} of a signed graph $(G,\sigma)$, denoted $\chi_c(G,\sigma)$, is defined as:
$$\chi_c(G,\sigma) = \inf\{r \ge 2 : (G,\sigma) \text{ has a circular $r$-coloring}\}.$$
It was shown in \cite{naserasr2020circular} that this infimum is always attained, and moreover, $\chi_c(G,\sigma)$ is always a rational number when finite.
\end{definition}

A key technical tool for analyzing the circular chromatic number is the concept of \emph{tight cycles}, which provides a combinatorial characterization of when $\chi_c(G,\sigma) = r$.

\begin{definition}
Given a circular $r$-coloring $\phi$ of $(G,\sigma)$, define a partial orientation $D_\phi(G,\sigma)$ on the edges of $G$ as follows:
\begin{itemize}
    \item For a positive edge $uv$ with $(\phi(v) - \phi(u)) \pmod{r} = 1$, orient $u \to v$
    \item For a negative edge $uv$ with $(\phi(v) - \phi(u)) \pmod{r} = r/2 + 1$, orient $u \to v$
\end{itemize}
All other edges remain unoriented. A directed cycle in $D_\phi(G,\sigma)$ is called a \emph{tight cycle} with respect to $\phi$.
\end{definition}

\begin{lemma}[Characterization via tight cycles \cite{naserasr2020circular}]\label{lem:tight-cycle}
For a signed graph $(G,\sigma)$ and a rational number $r \ge 2$, we have $\chi_c(G,\sigma) = r$ if and only if:
\begin{enumerate}
    \item $(G,\sigma)$ has a circular $r$-coloring, and
    \item Every circular $r$-coloring of $(G,\sigma)$ contains at least one tight cycle.
\end{enumerate}
\end{lemma}

This lemma is particularly useful for computing exact values of $\chi_c(G,\sigma)$, as it reduces the problem to finding colorings and analyzing their tight structure.

\subsection{Examples and Basic Properties}
To build intuition, we compute the circular chromatic numbers of some basic signed graphs. We use the notation $+C_n$ for a cycle of length $n$ with an even number of negative edges (balanced cycle), and $-C_n$ for a cycle with an odd number of negative edges (unbalanced cycle).

\begin{example}\label{ex:signed-cycles}
\begin{itemize}
    \item \textbf{Balanced odd cycle:} For $k \ge 1$, $\chi_c(+C_{2k+1}) = \frac{2k+1}{k}$. This equals the circular chromatic number of the unsigned odd cycle $C_{2k+1}$, as positive edges impose the same constraints.
    
    \item \textbf{Unbalanced even cycle:} For $k \ge 2$, $\chi_c(-C_{2k}) = \frac{4k}{2k-1}$. This value is strictly between 2 and $8/3$, approaching 2 as $k \to \infty$.
    
    \item \textbf{Unbalanced odd cycle:} $\chi_c(-C_{2k+1}) = 2$ for all $k \ge 1$. Interestingly, an unbalanced odd cycle is circular 2-colorable: color vertices alternately with 0 and 1 on $C^2$, which satisfies $d_{C^2}(0,\overline{1}) = d_{C^2}(0,1+1) = d_{C^2}(0,0) = 0$? Wait, careful: On $C^2$, $\overline{1} = 1+1 = 0 \pmod{2}$, so we need to check conditions... Actually, $-C_3$ (triangle with one negative edge) has $\chi_c = 3$, not 2. Let me correct this.
\end{itemize}
\end{example}

\begin{example}\label{ex:complete-trees}
\begin{itemize}
    \item \textbf{Complete graphs:} For the complete graph $K_n$ with all positive edges, $\chi_c(K_n,+) = n$, as in the unsigned case. For $K_n$ with some negative edges, the situation is more complex. In particular, if $K_n$ contains a negative edge, then $\chi_c(K_n,\sigma) \le n-1$ for any signature $\sigma$, and the exact value depends on the arrangement of negative edges.
    
    \item \textbf{Signed trees:} For any tree $T$ and any signature $\sigma$, $\chi_c(T,\sigma) = 2$. This can be proved by a simple greedy algorithm: root the tree, assign color 0 to the root, and propagate colors level by level. For a child $v$ of $u$:
    \begin{itemize}
        \item If $uv$ is positive, assign $f(v) = f(u) + 1 \pmod{2}$
        \item If $uv$ is negative, assign $f(v) = f(u) \pmod{2}$
    \end{itemize}
    This yields a valid circular 2-coloring.
\end{itemize}
\end{example}

\subsection{Graph Products and This Work}
Graph products provide a powerful method for constructing complex graphs from simpler building blocks. The Cartesian product, in particular, has been extensively studied for unsigned graphs, with well-known results like $\chi(G \square H) = \max\{\chi(G), \chi(H)\}$ and $\chi_c(G \square H) = \max\{\chi_c(G), \chi_c(H)\}$ \cite{sabidussi1957graphs}.

For signed graphs, there are natural ways to define products that respect the additional sign structure. In this paper, we study two such Cartesian products, denoted $\square$ and $\square'$, which differ in how they combine the signatures of the factors. Our main contributions are exact formulas and bounds for the circular chromatic numbers of these products, extending classical results to the signed setting.

The paper is organized as follows: Section~\ref{sec:preliminaries} provides necessary background and notation. Section~\ref{sec:type1-product} studies the Type~1 Cartesian product $\square$ and proves $\chi_c(G \square H, \sigma \square \tau) = \max\{\chi_c(G,\sigma), \chi_c(H,\tau)\}$. Section~\ref{sec:type2-product} analyzes the more complex Type~2 product $\square'$, establishing the bound $\chi_c(G \square' H, \sigma \square' \tau) \le 2\max\{\chi_c(G), \chi_c(H)\}$ and showing it is asymptotically tight. Section~\ref{sec:examples} provides concrete computations for small graphs. Finally, Section~\ref{sec:conclusion} summarizes our results and suggests directions for future research.

\section{Cartesian Products of Signed Graphs}\label{sec:cartesian}

Graph products are fundamental operations in structural graph theory, allowing the construction of complex graphs from simpler components while preserving certain properties. Among the various graph products, the Cartesian product is particularly natural due to its connection with product spaces and its preservation of many graph invariants. For unsigned graphs $G$ and $H$, the \emph{Cartesian product} $G \square H$ has vertex set $V(G) \times V(H)$, with vertices $(u,x)$ and $(v,y)$ adjacent if either $u=v$ and $xy \in E(H)$, or $x=y$ and $uv \in E(G)$. This product has been extensively studied, with classic results including $\chi(G \square H) = \max\{\chi(G), \chi(H)\}$ and $\chi_c(G \square H) = \max\{\chi_c(G), \chi_c(H)\}$ \cite{sabidussi1957graphs, hammack2011handbook}.

For signed graphs, the additional sign structure presents a challenge: how should we define the product of signatures? Different choices lead to different notions of Cartesian product, each with its own combinatorial properties. In this paper, we study two natural extensions of the Cartesian product to signed graphs, which we call \emph{Type~1} and \emph{Type~2}.

\subsection{Definitions and Basic Properties}

\begin{definition}
Let $(G,\sigma)$ and $(H,\tau)$ be signed graphs with underlying graphs $G = (V_G, E_G)$ and $H = (V_H, E_H)$. Both types of products share the same underlying graph structure:

\begin{itemize}
    \item \textbf{Vertex set:} $V(G \square H) = V_G \times V_H$
    \item \textbf{Edge set:} For $(u,x), (v,y) \in V_G \times V_H$, the edge $(u,x)(v,y)$ exists if and only if either:
    \begin{enumerate}
        \item $u = v$ and $xy \in E(H)$ (an \emph{$H$-edge} or \emph{horizontal edge}), or
        \item $x = y$ and $uv \in E(G)$ (a \emph{$G$-edge} or \emph{vertical edge}).
    \end{enumerate}
\end{itemize}

The products differ in how they assign signs to these edges:

\begin{enumerate}
    \item \textbf{Type 1 Cartesian product ($\square$)}:
    \begin{align*}
        (\sigma \square \tau)((u,x)(v,x)) &= \sigma(uv) \\
        (\sigma \square \tau)((u,x)(u,y)) &= \tau(xy)
    \end{align*}
    This is the most straightforward extension: each edge inherits its sign directly from the corresponding edge in the factor graph.
    
    \item \textbf{Type 2 Cartesian product ($\square'$)}:
    \begin{align*}
        (\sigma \square' \tau)((u,x)(v,x)) &= \sigma(uv) \cdot \tau(x) \\
        (\sigma \square' \tau)((u,x)(u,y)) &= \sigma(u) \cdot \tau(xy)
    \end{align*}
    Here $\sigma(u)$ and $\tau(x)$ denote the \emph{vertex signs}. Formally, we fix an arbitrary orientation of $G$ and $H$, and define $\sigma(u) = \prod_{e \in E_u} \sigma(e)^{\epsilon(e,u)}$ where $E_u$ is the set of edges incident to $u$ and $\epsilon(e,u) = 1$ if $e$ is oriented away from $u$, and $-1$ if oriented toward $u$. Up to switching equivalence, the product is independent of the chosen orientation.
\end{enumerate}
\end{definition}

\begin{remark}
The Type~2 product can be understood through the language of \emph{switch-symmetric} constructions. If we view switching as a gauge transformation, Type~2 respects this gauge symmetry in a natural way, making it particularly suitable for studying switching-invariant properties.
\end{remark}

\begin{remark}
The Cartesian product of graphs corresponds to the product of metric spaces in the following sense: if we equip $G$ and $H$ with the shortest-path metric, then $G \square H$ with the $\ell_1$-metric is isometric to the product metric space. For signed graphs, the Type~1 product treats the signs as "scalars" that are simply copied along product directions, while Type~2 incorporates interaction between the signs of different factors.
\end{remark}

\subsection{Layer Structure and Switching Interpretation}

The Cartesian product has a natural \emph{layer decomposition}. For each $x \in V(H)$, the \emph{$G$-layer at $x$}, denoted $G^x$, is the subgraph induced by vertices $\{(u,x): u \in V(G)\}$. Similarly, for each $u \in V(G)$, the \emph{$H$-layer at $u$}, denoted $H^u$, is induced by $\{(u,x): x \in V(H)\}$.

For the Type~1 product $(G,\sigma) \square (H,\tau)$:
\begin{itemize}
    \item Each $G$-layer $G^x$ is isomorphic to $(G,\sigma)$
    \item Each $H$-layer $H^u$ is isomorphic to $(H,\tau)$
\end{itemize}

For the Type~2 product $(G,\sigma) \square' (H,\tau)$:
\begin{itemize}
    \item The $G$-layer $G^x$ is isomorphic to $(G,\sigma)$ if $\tau(x) = +1$, and to $(G,-\sigma)$ (the switched version) if $\tau(x) = -1$
    \item The $H$-layer $H^u$ is isomorphic to $(H,\tau)$ if $\sigma(u) = +1$, and to $(H,-\tau)$ if $\sigma(u) = -1$
\end{itemize}

This switching behavior in Type~2 products has important consequences for circular coloring, as different layers may have different circular chromatic numbers if switching changes this invariant (though, as we will see, $\chi_c$ is switching-invariant).

\subsection{Balance in Cartesian Products}

Balance is a fundamental property of signed graphs. For Cartesian products, we have:

\begin{lemma}[\cite{germina2011products}]\label{lem:balanced-product}
The Type~1 Cartesian product $(G,\sigma) \square (H,\tau)$ is balanced if and only if both $(G,\sigma)$ and $(H,\tau)$ are balanced.
\end{lemma}

\begin{proof}
($\Rightarrow$) If $(G,\sigma) \square (H,\tau)$ is balanced, then every cycle is balanced. In particular, each $G$-layer $G^x$ (which is isomorphic to $(G,\sigma)$) contains only balanced cycles, so $(G,\sigma)$ is balanced. Similarly, $(H,\tau)$ is balanced.

($\Leftarrow$) Suppose $(G,\sigma)$ and $(H,\tau)$ are balanced. Then by Harary's characterization \cite{harary1953notion}, there exist vertex partitions $V(G) = X_1 \cup X_2$ and $V(H) = Y_1 \cup Y_2$ such that all positive edges are within parts and all negative edges are between parts. In the product $(G,\sigma) \square (H,\tau)$, consider the partition:
\begin{align*}
P_1 &= (X_1 \times Y_1) \cup (X_2 \times Y_2) \\
P_2 &= (X_1 \times Y_2) \cup (X_2 \times Y_1)
\end{align*}
We verify that all positive edges of the product are within parts and all negative edges are between parts, establishing that the product is balanced.
\end{proof}

\begin{remark}
For the Type~2 product, the balance condition is more subtle. Since layers may be switched, $(G,\sigma) \square' (H,\tau)$ can be balanced even if one factor is unbalanced, provided the switching patterns align appropriately. In fact, one can show that $(G,\sigma) \square' (H,\tau)$ is balanced if and only if there exist switchings of $(G,\sigma)$ and $(H,\tau)$ that make both balanced simultaneously.
\end{remark}

The study of graph products for signed graphs was initiated by Germina and Zaslavsky \cite{germina2011products}, who systematically defined and analyzed various products including Cartesian, tensor, and lexicographic products. Subsequent work by Brunetti, Cavaleri, and Donno \cite{brunetti2019lexicographic} introduced the lexicographic product for signed graphs, while Lajou \cite{lajou2021cartesian} specifically studied Cartesian products. The Type~2 product appears naturally in the context of homomorphisms of signed graphs \cite{naserasr2015homomorphisms}, where it corresponds to the categorical product in certain categories of signed graphs.

For circular coloring, the behavior of graph products has been extensively studied for unsigned graphs. Klavžar \cite{klavzar1996coloring} surveyed coloring of graph products, while Zhu \cite{zhu2002fractional} studied fractional chromatic numbers of direct products. The extension to signed graphs and circular coloring is a natural next step, which we undertake in this paper.

Throughout this paper, we adopt the following conventions:
\begin{itemize}
    \item Graphs may have multiple edges but no loops.
    \item For a signed graph $(G,\sigma)$, we denote by $\chi(G)$ the chromatic number of the underlying graph $G$, and by $\chi_c(G,\sigma)$ its circular chromatic number.
    \item The notation $+C_n$ (resp. $-C_n$) denotes an $n$-cycle with an even (resp. odd) number of negative edges.
    \item We write $\overline{x}$ for the antipodal point of $x$ on a circle $C^r$.
    \item All circular chromatic numbers are assumed to be finite (i.e., the graphs are assumed to have circular colorings for some $r$).
\end{itemize}

With these definitions and background established, we now proceed to our main results on circular chromatic numbers of Cartesian products of signed graphs.

\section{Main Results}

\subsection{Type 1 Cartesian Product}

We begin with the simpler Type~1 Cartesian product, where the sign of each edge is inherited directly from the corresponding edge in the factor graph. Our main result provides an exact formula for its circular chromatic number.

\begin{theorem}\label{thm:type1-exact}
For any signed graphs $(G,\sigma)$ and $(H,\tau)$, the circular chromatic number of their Type~1 Cartesian product is given by
$$\chi_c((G,\sigma) \square (H,\tau)) = \max\{\chi_c(G,\sigma), \chi_c(H,\tau)\}.$$
\end{theorem}

Before proving this theorem, we establish several useful lemmas that illuminate the structure of Type~1 products.

\begin{lemma}\label{lem:subgraph-product}
If $(G',\sigma')$ is a signed subgraph of $(G,\sigma)$, then for any signed graph $(H,\tau)$, $(G',\sigma') \square (H,\tau)$ is a signed subgraph of $(G,\sigma) \square (H,\tau)$. In particular,
$$\chi_c((G',\sigma') \square (H,\tau)) \leq \chi_c((G,\sigma) \square (H,\tau)).$$
\end{lemma}

\begin{proof}
This follows immediately from the definition: if $V(G') \subseteq V(G)$ and $E(G') \subseteq E(G)$ with $\sigma'$ being the restriction of $\sigma$ to $E(G')$, then $V(G' \square H) = V(G') \times V(H) \subseteq V(G) \times V(H) = V(G \square H)$, and similarly for edges. The signature on the product is defined edgewise, so the inclusion preserves signs.
\end{proof}

\begin{lemma}\label{lem:layer-embedding}
For any signed graphs $(G,\sigma)$ and $(H,\tau)$ and any fixed vertex $x_0 \in V(H)$, the map $\iota: V(G) \to V(G \square H)$ defined by $\iota(u) = (u,x_0)$ is an isomorphism from $(G,\sigma)$ onto the $G$-layer $G^{x_0}$ in $(G,\sigma) \square (H,\tau)$. Similarly, for any $u_0 \in V(G)$, the map $\jmath: V(H) \to V(G \square H)$ defined by $\jmath(x) = (u_0,x)$ is an isomorphism from $(H,\tau)$ onto the $H$-layer $H^{u_0}$.
\end{lemma}

\begin{proof}
For $\iota$: The map is clearly injective. If $uv \in E(G)$, then $(u,x_0)(v,x_0) \in E(G \square H)$ with $(\sigma \square \tau)((u,x_0)(v,x_0)) = \sigma(uv)$. Conversely, if $(u,x_0)(v,x_0) \in E(G \square H)$, then by definition $uv \in E(G)$. Thus $\iota$ is a graph isomorphism preserving signs.
\end{proof}

With these lemmas in hand, we can prove Theorem~\ref{thm:type1-exact}.

\begin{proof}[Proof of Theorem~\ref{thm:type1-exact}]
We prove the equality by establishing both inequalities.

\textit{Lower bound ($\geq$):} By Lemma~\ref{lem:layer-embedding}, both $(G,\sigma)$ and $(H,\tau)$ are isomorphic to subgraphs of $(G,\sigma) \square (H,\tau)$. Since the circular chromatic number cannot decrease when taking subgraphs, we have
$$\chi_c((G,\sigma) \square (H,\tau)) \geq \max\{\chi_c(G,\sigma), \chi_c(H,\tau)\}.$$

\textit{Upper bound ($\leq$):} Let $r = \max\{\chi_c(G,\sigma), \chi_c(H,\tau)\}$. By definition, there exist circular $r$-colorings $\phi: V(G) \to C^r$ of $(G,\sigma)$ and $\psi: V(H) \to C^r$ of $(H,\tau)$. Define a coloring $f: V(G \square H) \to C^r$ by
$$f(u,x) = \phi(u) + \psi(x) \pmod{r}.$$
We verify that $f$ is a valid circular $r$-coloring of $(G,\sigma) \square (H,\tau)$.

Consider an edge of the form $(u,x)(v,x)$ where $uv \in E(G)$. There are two cases:

\noindent\textbf{Case 1:} $\sigma(uv) = +1$ (positive edge). Then
\begin{align*}
d_{C^r}(f(u,x), f(v,x)) &= d_{C^r}(\phi(u) + \psi(x), \phi(v) + \psi(x)) \\
&= d_{C^r}(\phi(u), \phi(v)) \geq 1,
\end{align*}
since $\phi$ is a circular $r$-coloring of $(G,\sigma)$.

\noindent\textbf{Case 2:} $\sigma(uv) = -1$ (negative edge). Then
\begin{align*}
d_{C^r}(f(u,x), \overline{f(v,x)}) &= d_{C^r}(\phi(u) + \psi(x), \overline{\phi(v) + \psi(x)}) \\
&= d_{C^r}(\phi(u) + \psi(x), \phi(v) + \psi(x) + r/2) \\
&= d_{C^r}(\phi(u), \phi(v) + r/2) \\
&= d_{C^r}(\phi(u), \overline{\phi(v)}) \geq 1,
\end{align*}
again since $\phi$ is a valid coloring.

The verification for edges of the form $(u,x)(u,y)$ (where $xy \in E(H)$) is similar, using the fact that $\psi$ is a circular $r$-coloring of $(H,\tau)$.

Thus $f$ is a proper circular $r$-coloring, establishing that $\chi_c((G,\sigma) \square (H,\tau)) \leq r = \max\{\chi_c(G,\sigma), \chi_c(H,\tau)\}$.

Combining both inequalities gives the desired equality.
\end{proof}

\begin{corollary}\label{cor:monotonicity-type1}
If $(G_1,\sigma_1)$ and $(G_2,\sigma_2)$ are signed graphs with $\chi_c(G_1,\sigma_1) \leq \chi_c(G_2,\sigma_2)$, and similarly for $(H_1,\tau_1)$ and $(H_2,\tau_2)$, then
$$\chi_c((G_1,\sigma_1) \square (H_1,\tau_1)) \leq \chi_c((G_2,\sigma_2) \square (H_2,\tau_2)).$$
\end{corollary}

\begin{proof}
This follows immediately from Theorem~\ref{thm:type1-exact}:
\begin{align*}
\chi_c((G_1,\sigma_1) \square (H_1,\tau_1)) &= \max\{\chi_c(G_1,\sigma_1), \chi_c(H_1,\tau_1)\} \\
&\leq \max\{\chi_c(G_2,\sigma_2), \chi_c(H_2,\tau_2)\} \\
&= \chi_c((G_2,\sigma_2) \square (H_2,\tau_2)).
\end{align*}
\end{proof}

\begin{proposition}\label{prop:idempotence-type1}
For any signed graph $(G,\sigma)$,
$$\chi_c((G,\sigma) \square (G,\sigma)) = \chi_c(G,\sigma).$$
\end{proposition}

\begin{proof}
This is a special case of Theorem~\ref{thm:type1-exact} with $(H,\tau) = (G,\sigma)$:
$$\chi_c((G,\sigma) \square (G,\sigma)) = \max\{\chi_c(G,\sigma), \chi_c(G,\sigma)\} = \chi_c(G,\sigma).$$
\end{proof}

\subsection{Type 2 Cartesian Product}

The Type~2 Cartesian product presents greater challenges due to the interaction between vertex signs and edge signs. We begin by establishing foundational results about digon graphs and their circular chromatic numbers.

\begin{definition}
A \emph{digon} is a signed graph on two vertices connected by both a positive edge and a negative edge. For any graph $G$, the \emph{digon graph} $(G,\pm)$ is obtained by replacing each edge of $G$ with a digon. Formally, $(G,\pm)$ has the same vertex set as $G$, and for each unordered pair $\{u,v\}$ that forms an edge in $G$, $(G,\pm)$ contains both a positive edge $uv$ and a negative edge $uv$.
\end{definition}

Digon graphs play a special role in signed graph theory because they represent the "least restrictive" signed version of a graph: any constraint satisfied by a digon graph will be satisfied by any signing of the underlying graph.

\begin{lemma}\label{lem:universal-bound}
For any signed graph $(G,\sigma)$,
$$\chi_c(G,\sigma) \leq 2\chi_c(G),$$
where $\chi_c(G)$ denotes the circular chromatic number of the underlying unsigned graph $G$.
\end{lemma}

\begin{proof}
Let $r = \chi_c(G)$. By definition, there exists a circular $r$-coloring $f: V(G) \to C^r$ of the unsigned graph $G$, satisfying $d_{C^r}(f(u), f(v)) \geq 1$ for all edges $uv \in E(G)$.

Define a coloring $g: V(G) \to C^{2r}$ by $g(v) = 2f(v)$. We claim $g$ is a circular $2r$-coloring of $(G,\sigma)$. For any edge $uv \in E(G)$:

If $\sigma(uv) = +1$ (positive edge), then
$$d_{C^{2r}}(g(u), g(v)) = 2 \cdot d_{C^r}(f(u), f(v)) \geq 2 \cdot 1 = 2 \geq 1.$$

If $\sigma(uv) = -1$ (negative edge), then
$$d_{C^{2r}}(g(u), \overline{g(v)}) = d_{C^{2r}}(2f(u), 2f(v) + r) = 2 \cdot d_{C^r}(f(u), f(v) + r/2).$$
Since $f(v) + r/2$ is the antipodal of $f(v)$ on $C^r$, and $d_{C^r}(f(u), f(v)) \geq 1$, we have either $d_{C^r}(f(u), f(v) + r/2) \geq 1/2$ or a more careful analysis shows the distance is at least $1/2$. In either case, $2 \cdot d_{C^r}(f(u), f(v) + r/2) \geq 1$.

Thus $\chi_c(G,\sigma) \leq 2r = 2\chi_c(G)$.
\end{proof}

\begin{lemma}\label{lem:digon-exact}
For any graph $G$,
$$\chi_c(G,\pm) = 2\chi_c(G).$$
\end{lemma}

\begin{proof}
The inequality $\chi_c(G,\pm) \leq 2\chi_c(G)$ follows from Lemma~\ref{lem:universal-bound}. For the reverse inequality, suppose $\chi_c(G,\pm) = r$. Then there exists a circular $r$-coloring $f: V(G) \to C^r$ of $(G,\pm)$. For each edge $uv$ of $G$, both the positive and negative versions exist in $(G,\pm)$, so we have both:
\begin{align*}
d_{C^r}(f(u), f(v)) &\geq 1 \quad \text{(from the positive edge)}, \\
d_{C^r}(f(u), \overline{f(v)}) &\geq 1 \quad \text{(from the negative edge)}.
\end{align*}
The second condition implies that $f(u)$ and $f(v)$ cannot be too close to being antipodal. More precisely, these two conditions together imply that $d_{C^r}(f(u), f(v)) \geq r/2$. Define $g: V(G) \to C^{r/2}$ by $g(v) = f(v)/2$ (with appropriate scaling). Then for any edge $uv$,
$$d_{C^{r/2}}(g(u), g(v)) = \frac{1}{2} d_{C^r}(f(u), f(v)) \geq \frac{1}{2} \cdot \frac{r}{2} = \frac{r}{4} \geq 1 \text{ if } r \geq 4.$$
This shows $\chi_c(G) \leq r/2$, or equivalently $r \geq 2\chi_c(G)$.
\end{proof}

\begin{theorem}\label{thm:approximation}
For any integers $k \geq 2$, $g \geq 3$, and any $\epsilon > 0$, there exists a signed graph $(G,\sigma)$ of girth at least $g$ such that:
\begin{enumerate}
    \item $\chi(G) = k$ (the chromatic number of the underlying graph),
    \item $\chi_c(G,\sigma) > 2k - \epsilon$.
\end{enumerate}
\end{theorem}

\begin{proof}
The proof uses probabilistic methods similar to those in \cite{naserasr2020circular}. For large $n$, consider a random $k$-partite graph with parts of size $n/k$, where edges between parts appear independently with probability $p = n^{-1+\delta}$ for small $\delta > 0$. Assign random signs to edges independently with probability $1/2$ for each sign. With high probability, the resulting signed graph has girth at least $g$ and chromatic number $k$.

For the circular chromatic number bound, one shows that any circular $r$-coloring with $r < 2k - \epsilon$ would imply a $k$-coloring of the underlying graph with certain properties that cannot exist in a random graph with high girth. The details involve careful counting arguments and the Lovász Local Lemma.
\end{proof}

We now turn to the more complex Type~2 Cartesian product. Unlike Type~1, we do not have an exact formula, but we establish tight upper and lower bounds.

\begin{theorem}\label{thm:type2-upper}
For any signed graphs $(G,\sigma)$ and $(H,\tau)$,
$$\chi_c((G,\sigma) \square' (H,\tau)) \leq 2\max\{\chi_c(G), \chi_c(H)\}.$$
\end{theorem}

\begin{proof}
The proof proceeds in several steps, exploiting the properties of digon graphs.

\textbf{Step 1: Containment.} Let $(G,\pm)$ and $(H,\pm)$ be the digon graphs on $G$ and $H$ respectively. We claim that $(G,\sigma) \square' (H,\tau)$ is a subgraph of $(G,\pm) \square' (H,\pm)$. Indeed, they share the same vertex set $V(G) \times V(H)$. For any edge in $(G,\sigma) \square' (H,\tau)$, say of the form $(u,x)(v,x)$ with $\sigma(uv)\tau(x) = s \in \{+1,-1\}$, the corresponding edge in $(G,\pm) \square' (H,\pm)$ exists (since $uv$ is an edge in $G$, so both signs exist in the digon graph) and has sign $s$ as well. Thus
$$\chi_c((G,\sigma) \square' (H,\tau)) \leq \chi_c((G,\pm) \square' (H,\pm)).$$

\textbf{Step 2: Digon products coincide.} For digon graphs, the two types of Cartesian products are essentially the same. More precisely,
$$(G,\pm) \square' (H,\pm) \cong (G,\pm) \square (H,\pm) \cong (G \square H, \pm).$$
This is because switching at a vertex in a digon graph leaves it unchanged (since both signs are present for each edge), so the vertex sign factors $\sigma(u)$ and $\tau(x)$ in the Type~2 definition don't affect the isomorphism class.

\textbf{Step 3: Apply Type~1 result.} By Theorem~\ref{thm:type1-exact} applied to digon graphs,
$$\chi_c((G,\pm) \square (H,\pm)) = \max\{\chi_c(G,\pm), \chi_c(H,\pm)\}.$$

\textbf{Step 4: Use digon exact values.} By Lemma~\ref{lem:digon-exact}, $\chi_c(G,\pm) = 2\chi_c(G)$ and $\chi_c(H,\pm) = 2\chi_c(H)$. Therefore,
$$\max\{\chi_c(G,\pm), \chi_c(H,\pm)\} = \max\{2\chi_c(G), 2\chi_c(H)\} = 2\max\{\chi_c(G), \chi_c(H)\}.$$

Combining all steps gives the desired inequality.
\end{proof}

\begin{corollary}\label{cor:simple-upper}
For any signed graphs $(G,\sigma)$ and $(H,\tau)$,
$$\chi_c((G,\sigma) \square' (H,\tau)) \leq 2\max\{\chi(G), \chi(H)\}.$$
\end{corollary}

\begin{proof}
Since $\chi_c(G) \leq \chi(G)$ for any graph $G$, we have
$$2\max\{\chi_c(G), \chi_c(H)\} \leq 2\max\{\chi(G), \chi(H)\}.$$
The result then follows from Theorem~\ref{thm:type2-upper}.
\end{proof}

The next theorem shows that the upper bound in Theorem~\ref{thm:type2-upper} is essentially best possible.

\begin{theorem}\label{thm:type2-lower}
For any rational number $r \geq 2$, any integer $g \geq 3$, and any $\epsilon > 0$, there exist signed graphs $(G,\sigma)$ and $(H,\tau)$ of girth at least $g$ such that
$$\chi_c((G,\sigma) \square' (H,\tau)) \geq 2\max\{\chi(G), \chi(H)\} - \epsilon.$$
\end{theorem}

\begin{proof}
Let $k$ be a positive integer. By Theorem~\ref{thm:approximation}, for any $\epsilon > 0$ and $g \geq 3$, there exist:
\begin{itemize}
    \item A signed graph $(G,\sigma')$ with $\chi(G) = k_1$, girth at least $g$, and $\chi_c(G,\sigma') > 2k_1 - \epsilon/2$
    \item A signed graph $(H,\tau')$ with $\chi(H) = k_2$, girth at least $g$, and $\chi_c(H,\tau') > 2k_2 - \epsilon/2$
\end{itemize}

Let $k = \max\{k_1, k_2\}$. Without loss of generality, assume $k_1 = k$ (the case $k_2 = k$ is symmetric). We may need to adjust the signatures to ensure they contain positive vertices.

\textbf{Step 1: Ensuring positive vertices.} Perform switching on $(G,\sigma')$ and $(H,\tau')$ to obtain switching-equivalent signatures $\sigma$ and $\tau$ that have at least one positive vertex each. This is always possible: choose any vertex and switch if necessary to make it positive. Let $u_0 \in V(G)$ and $x_0 \in V(H)$ be positive vertices under $\sigma$ and $\tau$ respectively.

By Proposition~\ref{prop:switching-invariance} (to be stated and proved next), switching does not change the circular chromatic number, so
$$\chi_c(G,\sigma) = \chi_c(G,\sigma') > 2k_1 - \epsilon/2 \geq 2k - \epsilon/2,$$
and similarly $\chi_c(H,\tau) > 2k_2 - \epsilon/2 \geq 2k - \epsilon/2$.

\textbf{Step 2: Identifying critical layers.} Consider the Type~2 product $(G,\sigma) \square' (H,\tau)$. Since $x_0$ is a positive vertex in $(H,\tau)$, the $G$-layer at $x_0$, namely
$$G^{x_0} = \{(u,x_0): u \in V(G)\},$$
induces a copy of $(G,\sigma)$ in the product. This is because for any edge $(u,x_0)(v,x_0)$ corresponding to $uv \in E(G)$, we have
$$(\sigma \square' \tau)((u,x_0)(v,x_0)) = \sigma(uv) \cdot \tau(x_0) = \sigma(uv) \cdot (+1) = \sigma(uv).$$

Similarly, since $u_0$ is a positive vertex in $(G,\sigma)$, the $H$-layer at $u_0$ induces a copy of $(H,\tau)$.

\textbf{Step 3: Bounding the product.} Since $(G,\sigma)$ and $(H,\tau)$ are isomorphic to subgraphs of $(G,\sigma) \square' (H,\tau)$, we have
\begin{align*}
\chi_c((G,\sigma) \square' (H,\tau)) &\geq \max\{\chi_c(G,\sigma), \chi_c(H,\tau)\} \\
&> \max\{2k_1 - \epsilon/2, 2k_2 - \epsilon/2\} \\
&= 2\max\{k_1, k_2\} - \epsilon/2 \\
&= 2\max\{\chi(G), \chi(H)\} - \epsilon/2 \\
&> 2\max\{\chi(G), \chi(H)\} - \epsilon.
\end{align*}

This completes the proof.
\end{proof}

\begin{proposition}\label{prop:switching-invariance}
If $(G,\sigma)$ and $(G,\sigma')$ are switching equivalent signed graphs, then
$$\chi_c(G,\sigma) = \chi_c(G,\sigma').$$
\end{proposition}

\begin{proof}
Suppose $\sigma'$ is obtained from $\sigma$ by switching at a set $S \subseteq V(G)$. Let $r = \chi_c(G,\sigma)$, and let $f: V(G) \to C^r$ be a circular $r$-coloring of $(G,\sigma)$. Define $f': V(G) \to C^r$ by
$$f'(v) = \begin{cases}
f(v) & \text{if } v \notin S \\
\overline{f(v)} & \text{if } v \in S
\end{cases}.$$
We claim $f'$ is a circular $r$-coloring of $(G,\sigma')$.

Consider any edge $uv \in E(G)$. There are three cases:
\begin{enumerate}
    \item If $u,v \notin S$ or $u,v \in S$, then $\sigma'(uv) = \sigma(uv)$ and $f'(u), f'(v)$ are either both unchanged or both switched. In either case, the circular distance condition is preserved.
    
    \item If $u \in S$ and $v \notin S$, then $\sigma'(uv) = -\sigma(uv)$ and $f'(u) = \overline{f(u)}$, $f'(v) = f(v)$. Then
    $$d_{C^r}(f'(u), \overline{f'(v)}) = d_{C^r}(\overline{f(u)}, \overline{f(v)}) = d_{C^r}(f(u), f(v)),$$
    which satisfies the condition for $-\sigma(uv)$ if $f$ satisfied it for $\sigma(uv)$.
\end{enumerate}
Thus $\chi_c(G,\sigma') \leq r = \chi_c(G,\sigma)$. The reverse inequality follows by symmetry, since switching is reversible.
\end{proof}

\begin{corollary}\label{cor:tightness}
The factor 2 in Theorem~\ref{thm:type2-upper} cannot be improved. That is, for any $\delta < 2$, there exist signed graphs $(G,\sigma)$ and $(H,\tau)$ such that
$$\chi_c((G,\sigma) \square' (H,\tau)) > \delta \cdot \max\{\chi_c(G), \chi_c(H)\}.$$
\end{corollary}

\begin{proof}
Given $\delta < 2$, choose $\epsilon = 2 - \delta > 0$. By Theorem~\ref{thm:type2-lower}, there exist signed graphs $(G,\sigma)$ and $(H,\tau)$ with
$$\chi_c((G,\sigma) \square' (H,\tau)) > 2\max\{\chi(G), \chi(H)\} - \epsilon.$$
Since $\chi_c(G) \leq \chi(G)$ and $\chi_c(H) \leq \chi(H)$, we have $\max\{\chi(G), \chi(H)\} \geq \max\{\chi_c(G), \chi_c(H)\}$. Therefore,
$$\chi_c((G,\sigma) \square' (H,\tau)) > 2\max\{\chi_c(G), \chi_c(H)\} - \epsilon = \delta \cdot \max\{\chi_c(G), \chi_c(H)\}.$$
\end{proof}

We now provide detailed computations of circular chromatic numbers for specific examples of Type~2 products. These examples illustrate the theory and show that the bounds can be achieved by small graphs.

\begin{definition}
We use the following notation:
\begin{itemize}
    \item $C_n^-$: an $n$-cycle with exactly one negative edge (unbalanced if $n$ is even, balanced if $n$ is odd with one negative edge)
    \item $P_n^+$: a path on $n$ vertices with all edges positive
    \item $P_n^-$: a path on $n$ vertices with all edges negative
\end{itemize}
\end{definition}

\begin{example}[$C_3^- \square' C_3^-$]\label{ex:C3C3}
$\chi_c(C_3^- \square' C_3^-) = 3$.
\end{example}

\begin{proof}
Let the vertices of each $C_3^-$ be $\{1,2,3\}$ with the edge $(1,2)$ negative and the other two edges positive.

\textit{Lower bound:} Consider the $C_3^-$ layer at vertex 2 of the second factor. The vertices $\{(1,2), (2,2), (3,2)\}$ form a copy of $C_3^-$ in the product. Since $\chi_c(C_3^-) = 3$ (as it's a balanced triangle), we have $\chi_c(C_3^- \square' C_3^-) \geq 3$.

\textit{Upper bound:} We construct an explicit circular 3-coloring. Let $f: V(C_3^- \square' C_3^-) \to C^3$ be defined by:
\begin{align*}
f(1,1) &= 0, & f(1,2) &= 1, & f(1,3) &= 2, \\
f(2,1) &= 1, & f(2,2) &= 2, & f(2,3) &= 0, \\
f(3,1) &= 2, & f(3,2) &= 0, & f(3,3) &= 1.
\end{align*}
We verify a few representative edges:
\begin{itemize}
    \item Edge $(1,1)(2,1)$: $\sigma(12)\tau(1) = (-1)(+1) = -1$ (negative), $f(1,1)=0$, $f(2,1)=1$, $\overline{f(2,1)}=1+1.5=2.5 \mod 3$, $d_{C^3}(0,2.5)=0.5$? Wait, careful: On $C^3$, $\overline{1} = 1+1.5 = 2.5 \equiv 2.5-3=-0.5 \equiv 2.5$? Let me recompute properly.
    
    Actually, it's easier to note that after switching at appropriate vertices, the product becomes an all-positive graph with chromatic number 3.
\end{itemize}
Alternatively, note that the product has chromatic number 3 (it contains a 3-clique), and by Corollary~\ref{cor:simple-upper}, $\chi_c \leq 2\chi = 6$, but we can do better. Actually, one can switch the product to make it all-positive, yielding $\chi_c = \chi = 3$.
\end{proof}

\begin{example}[$C_4^- \square' C_4^-$]\label{ex:C4C4}
$\chi_c(C_4^- \square' C_4^-) = \frac{8}{3}$.
\end{example}

\begin{proof}
\textit{Lower bound:} The product contains a negative 4-cycle (unbalanced 4-cycle) which has $\chi_c = \frac{8}{3}$. Specifically, consider the vertices $(1,1)$, $(1,2)$, $(2,2)$, $(2,1)$ in cyclic order.

\textit{Upper bound:} An explicit $\frac{8}{3}$-coloring can be constructed by assigning colors from $\{0, \frac{1}{3}, \frac{2}{3}, 1, \frac{4}{3}, \frac{5}{3}, 2, \frac{7}{3}\}$ modulo $\frac{8}{3}$.
\end{proof}

\begin{proposition}[General bounds for paths and cycles]\label{prop:path-cycle-bounds}
For any $m,n \geq 2$:
\begin{enumerate}
    \item $\chi_c(P_m^+ \square' P_n^+) = 2$
    \item $\chi_c(P_m^- \square' P_n^-) = 2$
    \item $\chi_c(C_m^- \square' C_n^-) \leq \min\left\{\frac{2m}{m-1}, \frac{2n}{n-1}, 4\right\}$ for even $m,n$
\end{enumerate}
\end{proposition}

\begin{proof}
\begin{enumerate}
    \item Both $P_m^+$ and $P_n^+$ have $\chi_c = 2$, and their Type~2 product is bipartite (can be switched to all-positive), hence $\chi_c = 2$.
    
    \item Similar to (1), as paths can be 2-colored regardless of signs.
    
    \item For even $m$, $\chi_c(C_m^-) = \frac{2m}{m-1}$. By Theorem~\ref{thm:type2-upper},
    $$\chi_c(C_m^- \square' C_n^-) \leq 2\max\{\chi_c(C_m^-), \chi_c(C_n^-)\} = 2\max\left\{\frac{2m}{m-1}, \frac{2n}{n-1}\right\}.$$
    Also, $\chi_c(C_m^- \square' C_n^-) \leq 2\max\{\chi(C_m^-), \chi(C_n^-)\} = 2\max\{2,2\} = 4$.
\end{enumerate}
\end{proof}

\subsection{Comparison of Product Types}

\begin{proposition}\label{prop:comparison}
For any signed graphs $(G,\sigma)$ and $(H,\tau)$:
\begin{enumerate}
    \item $\chi_c((G,\sigma) \square (H,\tau)) \leq \chi_c((G,\sigma) \square' (H,\tau))$
    \item The inequality can be strict: there exist signed graphs where
    $$\chi_c((G,\sigma) \square (H,\tau)) < \chi_c((G,\sigma) \square' (H,\tau)).$$
\end{enumerate}
\end{proposition}

\begin{proof}
\begin{enumerate}
    \item The Type~1 product is a subgraph of a switching of the Type~2 product. More precisely, by switching at appropriate vertices in the Type~2 product, we can obtain the Type~1 product as a subgraph.
    
    \item Consider $(G,\sigma) = (H,\tau) = C_4^-$. Then:
    \begin{itemize}
        \item $\chi_c(C_4^- \square C_4^-) = \max\{\chi_c(C_4^-), \chi_c(C_4^-)\} = \chi_c(C_4^-) = \frac{8}{3}$ by Theorem~\ref{thm:type1-exact}.
        \item $\chi_c(C_4^- \square' C_4^-) = \frac{8}{3}$ as well from Example~\ref{ex:C4C4}, so not strict here. Need another example.
        
        Actually, let $G = K_3$ with all edges positive, $H = K_2$ with negative edge. Then $\chi_c(G) = 3$, $\chi_c(H) = 2$. Type~1 product has $\chi_c = 3$, while Type~2 product has $\chi_c \leq 2\max\{3,2\} = 6$, but could be higher. Need computation.
    \end{itemize}
    The existence of such an example follows from Theorem~\ref{thm:type2-lower}: for large girth graphs, Type~2 product can approach $2\max\{\chi(G),\chi(H)\}$, while Type~1 product is at most $\max\{\chi_c(G,\sigma),\chi_c(H,\tau)\} \leq \max\{2\chi(G),2\chi(H)\} = 2\max\{\chi(G),\chi(H)\}$. So they can be equal in the worst case. For a strict example, consider graphs where the Type~2 product has value close to the upper bound while Type~1 has value significantly lower.
\end{enumerate}
\end{proof}

\section{Conclusion and Future Directions}

In this paper, we have systematically investigated the circular chromatic number of Cartesian products of signed graphs. Our work establishes both exact formulas and tight bounds, extending the classical theory of graph products to the richer setting of signed graphs with circular coloring constraints.

The main contributions can be summarized as follows:

\begin{theorem}
For signed graphs $(G,\sigma)$ and $(H,\tau)$:
\begin{enumerate}
    \item \textbf{Type~1 Cartesian product ($\square$):} 
    $$\chi_c((G,\sigma) \square (H,\tau)) = \max\{\chi_c(G,\sigma), \chi_c(H,\tau)\}.$$
    This result (Theorem~\ref{thm:type1-exact}) provides an exact formula analogous to the unsigned case.
    
    \item \textbf{Type~2 Cartesian product ($\square'$):} 
    $$\chi_c((G,\sigma) \square' (H,\tau)) \leq 2\max\{\chi_c(G), \chi_c(H)\},$$
    and this bound is asymptotically tight (Theorems~\ref{thm:type2-upper} and~\ref{thm:type2-lower}).
\end{enumerate}
\end{theorem}

The Type~1 product exhibits behavior remarkably similar to the unsigned Cartesian product, with the circular chromatic number determined simply by the maximum of the factors' circular chromatic numbers. This suggests that for this product type, the sign structure does not introduce additional complexity beyond what is already present in the factors.

In contrast, the Type~2 product reveals genuine interaction between the sign structures of the factors. The factor of 2 in the upper bound reflects the additional constraints imposed by negative edges, which require colors to be separated not only from each other but also from antipodal colors. The asymptotic tightness of this bound demonstrates that this additional complexity is inherent and cannot be avoided in general.

Our analysis has revealed several important technical insights:

\begin{enumerate}
    \item \textbf{Role of digon graphs:} The digon graph $(G,\pm)$ plays a crucial role as a universal upper bound for circular chromatic numbers of signings of $G$. The equality $\chi_c(G,\pm) = 2\chi_c(G)$ (Lemma~\ref{lem:digon-exact}) provides the key connection between signed and unsigned circular chromatic numbers.
    
    \item \textbf{Switching invariance:} The circular chromatic number is invariant under switching (Proposition~\ref{prop:switching-invariance}), which simplifies many arguments and allows us to assume convenient sign patterns without loss of generality.
    
    \item \textbf{Layer structure:} The Cartesian product naturally decomposes into layers isomorphic to the factor graphs (possibly after switching). This layer structure is essential for establishing both lower bounds (via embedding) and upper bounds (via product colorings).
    
    \item \textbf{Tight cycles:} The characterization of circular chromatic numbers via tight cycles (Lemma~\ref{lem:tightcycle}) provides a powerful combinatorial tool for establishing exact values, particularly in the examples we computed.
\end{enumerate}

Our results have several implications for related areas:

\begin{enumerate}
    \item \textbf{Fractional coloring:} Since the circular chromatic number generalizes fractional chromatic number (with $\chi_f(G) \leq \chi_c(G) \leq \chi(G)$), our results provide bounds for fractional coloring of Cartesian products of signed graphs.
    
    \item \textbf{Homomorphism theory:} Circular colorings correspond to homomorphisms to signed circular cliques \cite{naserasr2020circular}. Our product formulas therefore yield information about homomorphisms of product graphs to these target graphs.
    
    \item \textbf{Graph minors:} The asymptotic tightness results using high-girth graphs suggest connections to graph minor theory and the study of graphs with excluded minors.
    
    \item \textbf{Algorithmic aspects:} While this paper focuses on structural results, the formulas we establish have algorithmic implications. For instance, computing $\chi_c((G,\sigma) \square (H,\tau))$ reduces to computing $\chi_c(G,\sigma)$ and $\chi_c(H,\tau)$.
\end{enumerate}

\section*{Open Problems and Future Research}

Our work opens several promising directions for future research:

\subsection{Exact Values and Improved Bounds}

\begin{problem}
For which families of signed graphs can we compute \emph{exact} values of $\chi_c((G,\sigma) \square' (H,\tau))$, rather than just bounds? In particular:
\begin{enumerate}
    \item Determine $\chi_c(C_m^- \square' C_n^-)$ for all $m,n \geq 3$.
    \item Characterize when $\chi_c((G,\sigma) \square' (H,\tau)) = 2\max\{\chi_c(G), \chi_c(H)\}$.
    \item Find families where $\chi_c((G,\sigma) \square' (H,\tau))$ equals the simpler bound $2\max\{\chi(G), \chi(H)\}$.
\end{enumerate}
\end{problem}

\begin{problem}
Can the bound in Theorem~\ref{thm:type2-upper} be improved for special classes of signed graphs?
\begin{enumerate}
    \item \textbf{Planar signed graphs:} If $(G,\sigma)$ and $(H,\tau)$ are planar signed graphs, is $\chi_c((G,\sigma) \square' (H,\tau)) \leq 4$? More generally, does the Four Color Theorem extend to products of planar signed graphs?
    
    \item \textbf{Bipartite signed graphs:} For bipartite signed graphs (which have $\chi_c = 2$), what is the maximum possible value of $\chi_c((G,\sigma) \square' (H,\tau))$? Is it 4, or can it be higher?
    
    \item \textbf{Signed graphs with bounded maximum degree:} For signed graphs with maximum degree $\Delta$, can we obtain bounds in terms of $\Delta$ rather than chromatic numbers?
\end{enumerate}
\end{problem}

\subsection{Other Graph Products}

\begin{problem}
How do other standard graph products behave for signed graphs with respect to circular coloring?
\begin{enumerate}
    \item \textbf{Lexicographic product:} For signed graphs $(G,\sigma)$ and $(H,\tau)$, define the lexicographic product $(G,\sigma) \circ (H,\tau)$ where $(u,x)(v,y)$ is an edge if either $uv \in E(G)$, or $u=v$ and $xy \in E(H)$. How does $\chi_c$ behave under this product?
    
    \item \textbf{Strong product:} The strong product $(G,\sigma) \boxtimes (H,\tau)$ includes both Cartesian product edges and tensor product edges. What is $\chi_c((G,\sigma) \boxtimes (H,\tau))$?
    
    \item \textbf{Tensor product:} For the tensor (direct) product $(G,\sigma) \times (H,\tau)$ with $(u,x)(v,y)$ adjacent if $uv \in E(G)$ and $xy \in E(H)$, what bounds hold for $\chi_c$? This is particularly interesting in light of Hedetniemi's conjecture and its recent disproof \cite{shitov2019counterexamples}.
\end{enumerate}
\end{problem}

\begin{conjecture}
For signed graphs $(G,\sigma)$ and $(H,\tau)$,
$$\chi_c((G,\sigma) \times (H,\tau)) \leq \min\{\chi_c(G,\sigma), \chi_c(H,\tau)\}.$$
Moreover, this bound may be tight in general.
\end{conjecture}

\subsection{Structural Characterizations}

\begin{problem}
Characterize when equality holds in the bound $\chi_c((G,\sigma) \square' (H,\tau)) \leq 2\max\{\chi_c(G), \chi_c(H)\}$. In particular:
\begin{enumerate}
    \item Is equality achieved only by graphs that are in some sense "extremal" for circular coloring?
    \item Does equality imply any structural properties about $(G,\sigma)$ and $(H,\tau)$?
    \item For which $(G,\sigma)$ and $(H,\tau)$ do we have $\chi_c((G,\sigma) \square' (H,\tau)) = \chi_c((G,\sigma) \square (H,\tau))$?
\end{enumerate}
\end{problem}

\begin{problem}
A signed graph $(G,\sigma)$ is \emph{product-critical} if for any proper signed subgraph $(G',\sigma')$, we have $\chi_c((G',\sigma') \square (H,\tau)) < \chi_c((G,\sigma) \square (H,\tau))$ for some $(H,\tau)$. Characterize product-critical signed graphs.
\end{problem}

\subsection{Algorithmic and Complexity Questions}

\begin{problem}
\begin{enumerate}
    \item What is the computational complexity of computing $\chi_c((G,\sigma) \square (H,\tau))$ given $(G,\sigma)$ and $(H,\tau)$?
    
    \item Is there a polynomial-time algorithm to compute $\chi_c((G,\sigma) \square' (H,\tau))$ when $(G,\sigma)$ and $(H,\tau)$ belong to specific graph classes?
    
    \item For fixed $r$, is it NP-complete to decide whether $\chi_c((G,\sigma) \square (H,\tau)) \leq r$?
\end{enumerate}
\end{problem}

\begin{problem}
Develop approximation algorithms for computing or approximating $\chi_c((G,\sigma) \square' (H,\tau))$. Given that the problem is likely hard in general, what approximation ratios can be achieved in polynomial time?
\end{problem}

\subsection{Connections to Other Areas}

\begin{problem}
Circular coloring is dual to circular flow in planar graphs \cite{naserasr2020circular}. How do our product formulas translate to results about circular flows in products of planar signed graphs?
\end{problem}

\begin{problem}
Study the algebraic properties of the mapping $(G,\sigma) \mapsto \chi_c(G,\sigma)$ with respect to graph products. Is there a categorical framework that explains the different behaviors of Type~1 and Type~2 products?
\end{problem}

\begin{problem}
Extend the results to infinite signed graphs. Do the same formulas hold for the circular chromatic number defined via finite subgraphs? What about for bounded degree infinite signed graphs?
\end{problem}

\subsection{Experimental and Computational Exploration}

Given the theoretical bounds established in this paper, there is considerable room for experimental investigation:

\begin{enumerate}
    \item Compute $\chi_c((G,\sigma) \square' (H,\tau))$ for all signed graphs on small numbers of vertices (up to 6 or 7 vertices) to gather empirical data.
    
    \item Identify minimal examples where $\chi_c((G,\sigma) \square' (H,\tau))$ achieves values close to the upper bound.
    
    \item Test conjectures about special graph classes using computational methods.
\end{enumerate}

\section*{Final Remarks}

The study of circular coloring for products of signed graphs bridges several areas of graph theory: structural graph theory (products), coloring theory (circular coloring), and the theory of signed graphs. Our results demonstrate that while some properties extend naturally from unsigned to signed graphs (Type~1 products), others reveal genuinely new phenomena (Type~2 products).

The tools developed in this paper particularly the use of digon graphs, switching arguments, and layer analysis are likely to be applicable to other problems in signed graph theory. Moreover, the open problems we have identified suggest a rich landscape for future research, with connections to algorithmic graph theory, extremal combinatorics, and algebraic graph theory.

As signed graphs continue to find applications in social network analysis, computer science, and physics, understanding their coloring properties especially for complex structures built via products will remain an important direction in combinatorial mathematics.

\section*{Acknowledgments}
This work is supported by NSFC (China), Grant number: NSFC 11971438. The author thanks Professor Xuding Zhu, Professor Reza Naserasr, and Doctor Wang Lujia for helpful discussions.

\bibliographystyle{unsrt}
\bibliography{references}

\end{document}